\documentclass[12pt,leqno,twoside,sort]{amsart}

\usepackage{amssymb,amsmath,amsthm,soul,color}

\usepackage[normalem]{ulem}

\usepackage{amssymb,amsmath,amsthm}

\usepackage[utf8]{inputenc}

\usepackage{todonotes}

\usepackage[left=2cm, right=2cm, top=2cm, bottom=2cm]{geometry}

\usepackage{url}

\usepackage[numbers,sort&compress]{natbib}

\usepackage{fixmath}

\usepackage[T1]{fontenc}
\usepackage{lmodern}
\usepackage{microtype}
\usepackage[normalem]{ulem}

\usepackage[colorlinks=true,urlcolor=blue,
citecolor=red,linkcolor=blue,linktocpage,pdfpagelabels,
bookmarksnumbered,bookmarksopen]{hyperref}

\newtheorem{Th}{Theorem}[section]
\newtheorem{Prop}[Th]{Proposition}
\newtheorem{Lem}[Th]{Lemma}

\newenvironment{altproof}[1]
{\noindent
{\em Proof of {#1}}.}
{\nopagebreak\mbox{}\hfill $\Box$\par\addvspace{0.5cm}}

\newcommand{\wt}{\widetilde}

\newcommand{\eps}{\varepsilon}

\newcommand{\R}{\mathbb{R}}

\newcommand{\Z}{\mathbb{Z}}

\newcommand{\cA}{{\mathcal A}}

\newcommand{\cC}{{\mathcal C}}
\newcommand{\cD}{{\mathcal D}}

\newcommand{\cG}{{\mathcal G}}
\newcommand{\cH}{{\mathcal H}}
\newcommand{\cI}{{\mathcal I}}
\newcommand{\cJ}{{\mathcal J}}
\newcommand{\cK}{{\mathcal K}}

\newcommand{\cM}{{\mathcal M}}
\newcommand{\cN}{{\mathcal N}}

\newcommand{\cP}{{\mathcal P}}
\newcommand{\cQ}{{\mathcal Q}}

\newcommand{\cT}{{\mathcal T}}

\newcommand{\ga}{\gamma}

\newcommand{\Ga}{\Gamma}
\newcommand{\Om}{\Omega}

\newcommand\J{\mathcal{J}}

\newcommand\tJ{\wt{\cJ}}

\newcommand{\weakto}{\rightharpoonup}

\newcommand{\cTto}{\stackrel{\cT}{\longrightarrow}}

\renewcommand\J{\mathcal{J}}

\numberwithin{equation}{section}

\newcommand{\supp}{\mathrm{supp}\,}

\newcommand{\loc}{\mathrm{loc}}

\begin{document}

\title{Bound states for the Schr\"odinger equation with mixed-type nonlinearites}

\author[B. Bieganowski]{Bartosz Bieganowski}
\address[B. Bieganowski]{\newline\indent  	Faculty of Mathematics and Computer Science,		\newline\indent 	Nicolaus Copernicus University, \newline\indent ul. Chopina 12/18, 87-100 Toru\'n, Poland}	\email{\href{mailto:bartoszb@mat.umk.pl}{bartoszb@mat.umk.pl}}	

\author[J. Mederski]{Jaros\l aw Mederski}
\address[J. Mederski]{\newline\indent  	Institute of Mathematics, \newline\indent		Polish Academy of Sciences, \newline\indent  ul. \'Sniadeckich 8, 00-956		Warszawa, Poland \newline\indent and	\newline\indent		Faculty of Mathematics and Computer Science,		\newline\indent 	Nicolaus Copernicus University, \newline\indent ul. Chopina 12/18, 87-100 Toru\'n, Poland}	\email{\href{mailto:jmederski@impan.pl}{jmederski@impan.pl}}

\date{}	\date{\today} 

\maketitle

\begin{abstract} 
We prove the existence results for the Schr\"odinger equation of the form
$$
-\Delta u + V(x) u = g(x,u), \quad x \in \R^N,
$$
where $g$ is superlinear and subcritical in some periodic set $K$ and linear in $\R^N \setminus K$ for sufficiently large $|u|$. The periodic potential $V$ is such that $0$ lies in a spectral gap of  $-\Delta+V$.
We find a solution with the energy bounded by a certain min-max level, and infinitely many geometrically distinct solutions provided that $g$ is odd in $u$.
\medskip

\noindent \textbf{Keywords:} variational methods, strongly indefinite functional, Cerami sequences, nonlinear Schr\"odinger equation, superlinear nonlinearity, Kerr effect, saturation effect, multiplicity of solutions.
   
\noindent \textbf{AMS Subject Classification:}  35Q55, 35Q60, 35A15, 35J20, 58E05
\end{abstract}

\section{Introduction}

The nonlinear Schr\"odinger equation
\begin{equation}\label{eq}
-\Delta u + V(x) u = g(x,u), \quad x \in \R^N, \ u \in H^1 (\R^N)
\end{equation}
arises in many various branches of mathematical physics, in particular  the so-called \textit{standing waves} $\Phi(x,t) = e^{-\mathrm{i}\omega t} u(x)$ of the time-dependent, nonlinear Schr\"odinger equation of the form
$$
\mathrm{i} \frac{\partial\Phi(x,t)}{\partial t} = - \Delta \Phi(x,t) + (V(x) + \omega ) \Phi(x,t) - g(x,|\Phi(x,t)|) \Phi(x,t), \quad (x,t) \in \R^N \times \R,
$$
appear in models in quantum physics. In nonlinear optics,  \eqref{eq} describes the propagation of a electromagnetic wave in a periodic waveguide, e.g. photonic crystals (\cite{Buryak, Kuchment, Pankov}). The external potential $V : \R^N \rightarrow \R$ takes into account the linear properties of the material and the nonlinear term $g : \R^N \times \R \rightarrow \R$ is responsible for the polarization of the medium. For instance, in {\em Kerr-like media} one has
$$
g(x,u) = \Gamma(x) |u|^2 u,
$$
and in the {\em saturation effect}, the nonlinear polarization is asymptotically linear and is of the form
$$
g(x,u) = \Gamma(x) \frac{|u|^2}{1+|u|^2}u.
$$
Recently it has been shown that materials with large range of prescribed properties can be created (\cite{Pendry, Shalaev, Smith,Dror,MalomedAzbel}) with different linear and nonlinear effects. Our aim is to model a wide range of nonlinear phenomena that allow to consider a composite of materials with different nonlinear polarization. In our case, the polarization $g(x,\cdot)$ may be linear for some $x\in\R^N\setminus K$ (for sufficiently large $|u|$) and nonlinear outside of it, where $K$ is a given $\Z^N$-periodic subset of $\R^N$. We admit also the fully nonlinear situation with $K=\R^N$.  In particular, we may combine the Kerr-like nonlinearity with a saturation effect, e.g.
$$
g(x,u) = \left\{ 
\begin{array}{ll}
\Gamma(x) |u|^2 u & \quad x \in K, \\
\chi_{\{|u| < 1\}} \frac{|u|^2}{1+|u|^2}u + \chi_{\{|u| \geq 1\}} \frac12 u & \quad x \in \R^N \setminus K,
\end{array} \right.
$$
where $\Gamma \in L^\infty (\R^N)$ is $\Z^N$-periodic, positive and bounded away from $0$ and $\chi$ stands for the characteristic function.

As usual we assume that the potential satisfies the following condition, cf. \cite{Pankov,ReedSimon,AlamaLi,CotiZelatiRabibowitz}:

\begin{enumerate}
\item[(V)] $V \in L^\infty (\R^N)$ is $\mathbb{Z}^N$-periodic and $0$ lies in the spectral gap of $-\Delta + V(x)$. 
\end{enumerate}
Recall that the spectrum of the operator $-\Delta + V(x)$ on $L^2 (\R^N)$, where $V \in L^\infty (\R^N)$ is $\mathbb{Z}^N$-periodic, is purely continuous and consists of pairwise disjoint, closed intervals (\cite{ReedSimon}). Thus we define that a spectral gap is any connected component of $\R \setminus \sigma (-\Delta + V(x))$.

Moreover we suppose that $g : \R^N \times \R \rightarrow \R$ is a Carath\'eodory function such that $x \mapsto g(x,u)$ is $\mathbb{Z}^N$-periodic for a.e. $x \in \R^N$ and for all $u \in \R$, i.e.
$
g(x+z, u) = g(x,u)$ for a.e.  $x \in \R^N$ and all $u \in \R$,
which satisfies the following conditions.
\begin{enumerate}
\item[(G1)] $g(x,u) = o(u)$ for $u \to 0$ uniformly in $x \in \R^N$.
\item[(G2)] There are $C > 0$ and $2 < p < 2^*$, where $2^*=\frac{2N}{N-2}$ for $N \geq 3$ and $2^* = \infty$ for $N \in \{1,2\}$, such that
$$
|g(x,u)| \leq C (1 + |u|^{p-1}) \quad \mbox{ for all } u \in \R \ \mbox{and a.e.} \ x \in \R^N.
$$
\item[(G3)] There exists closed and $\mathbb{Z}^N$-periodic subset $K \subset \R^N$ with $|K|>0$ such that
$$
\frac{G(x,u)}{u^2} \to \infty \quad \mbox{as } |u| \to \infty \mbox{ uniformly in } x \in K,
$$
where $G(x,u) := \int_0^u g(x,s) \, ds$, and $|\cdot|$ denotes the Lebesgue measure.
\item[(G4)] The function $\R \setminus \{0\} \ni u \mapsto \frac{g(x,u)}{|u|}\in \R$ is nondecreasing on $(-\infty,0)$ and on $(0,\infty)$  for a.e. $x \in \R^N$.
\item[(G5)]There is a function $\Theta \in L^\infty (\R^N \setminus K)$ and a constant $a > 0$ such that
$$
\frac{g(x,u)}{u} = \Theta (x) \quad \mbox{for} \ |u| \geq a > 0 \ \mbox{and a.e. } x \in \R^N \setminus K,
$$ 
and $\Theta$ is $\mathbb{Z}^N$-periodic.
\item[(G6)] $0$ is not an eigenvalue of  $-\Delta+V(x)-\Theta(x)$ on $L^2(\R^N\setminus K)$ with the Dirichlet boundary condition.
\end{enumerate}
In particular, (G1), (G4) and (G5) imply that
$$
0\leq \frac{g(x,u)}{u} \leq \Theta (x) \quad \mbox{for a.e.} \ x \in \R^N \setminus K.
$$
We look for weak solution to \eqref{eq} by means of critical points of the strongly indefinite energy functional $\cJ : H^1 (\R^N) \rightarrow \R$ is given by
$$
\cJ(u) := \frac{1}{2} \int_{\R^N} |\nabla u|^2 + V(x) u^2 \, dx - \int_{\R^N} G(x,u) \, dx.
$$
Conditions (G1), (G2), (G4) are standard and considered e.g. in \cite{SzulkinWeth,CotiZelatiRabibowitz,AlamaLi,Mederski2016,Liu} and in the references therein. Assumptions (G3) and (G5) are new and for large $|u|$, (G3) describes super-quadratic behaviour of $G$ in $K$ whereas (G5) provides quadratic behaviour of $G$ outside $K$. In a particular case $K=\R^N$, (G3) reduces to the usual super quadratic condition, (G5) is redundant and  there is a series of results concerning the existence and multiplicity of solutions in this case.
For instance, this situation has been recently considered  by Liu \cite{Liu}, Mederski \cite{Mederski2016}, De Paiva, Kryszewski and Szulkin \cite{dePaivaKryszewskiSzulkin}, and under stronger monotonicity assumption than (G4) by Szulkin and Weth in \cite{SzulkinWeth}. Their proofs are based on a generalized linking theorem that is applied to $\cJ$, or on the minimization method for $\cJ$ on the so-called Nehari-Pankov manifold. 
If $K$ is a proper subset of $\R^N$, then situation is much more delicate, since  in  general $\J(tu)$ may diverge to $+\infty$ or $-\infty$  as $t\to\infty$ for different $u$. It is easy to check that (G1) and (G4) imply that
\begin{equation}\label{AR2}
g(x,u)u \geq  2 G(x,u)  \quad \mbox{for all } u \in \R \mbox{ and for a.e. } x \in \R^N,
\end{equation}
however the Ambrosetti-Rabinowitz condition need not to be satisfied \cite{AR}, hence  we do not know whether any Palais-Smale sequence is bounded.

Observe that (G6) is clearly satisfied if $K=\R^N$. Suppose that $\R^N\setminus K$ is a periodically perforated medium in the following sense: $K=\bigcup_{z\in \Z^N}  (\overline{Q}+z)$, where  $Q$ is a smooth, contractible bounded domain such that $(\overline{Q}+z)\cap \overline{Q}=\emptyset$ for $z\in\Z^N\setminus\{0\}$. Then, in view of a recent result due to Kuchment and Taskinen \cite[Theorem 7]{KuchmentTaskinen}, the spectrum $\sigma(-\Delta+V(x)-\Theta(x))$ on $\R^N\setminus K$ with the Dirichlet boundary conditions is absolutely continuous, hence $-\Delta+V(x)-\Theta(x)$ does not admit any eigenvalue and (G6) is satisfied; see also \cite{KuchmentBAMS,Cardone} and references therein.

We would like to mention that the asymptotically linear case, e.g. the saturation effect, has also been intensively studied (see e.g. \cite{LiuSuWeth, StuartZhou,LiSzulkin,Maia} and references therein). In the positive-definite case, Jeanjean and Tanaka \cite{JeanjeanTanakaESAIM} provided an existence result for $g$ asymptotically linear. Maia, Oliviera and Ruviaro showed that for autonomous and asymptotically linear nonlinearity $g$ in the indefinite, non-periodic case there exists a nontrivial solution \cite{Maia}. Szulkin and Li showed that there is a nontrivial solution for $g$ asymptotically linear in the indefinite, $\Z^N$-periodic, nonautonomous case in \cite{LiSzulkin}.

Observe that, taking $\Theta \equiv 0$, we can consider a nonlinear term of the form
$$
g(x,u) = \chi_K (x) |u|^{p-2}u,
$$
where $2 < p < 2^*$ and $K$ satisfies the foregoing assumptions. An example of such a set is
$$
K = \bigcup_{z \in \mathbb{Z}^N} \left( [0,1/2] + z \right).
$$
In general, $|K| = \infty$. Indeed, since
$|\partial K| = 0$, we see that $|\mathrm{int}\, K| > 0$ and there is an open subset $\Om$ such that $\overline{\Om} \subset \mathrm{int}\, K\cap (0,1)^N$. Hence  $\overline{\Omega}+z \subset \mathrm{int}\, K$ for any $z \in \Z^N$.

Recall that $H^1(\R^N)$ has an orthogonal splitting $X^+\oplus X^-$ such that the quadratic form
$$
u\mapsto \int_{\R^N} |\nabla u|^2 + V(x) |u|^2 \, dx
$$
is positive definite on $X^+$ and negative definite on $X^-$ and if $\inf\sigma(-\Delta+V)<0$, then $\dim X^-=\infty$ and $\cJ$ is strongly indefinite. Similarly as in \cite{Pankov,SzulkinWeth} we introduce 
the so-called Nehari-Pankov manifold
$$
\cN := \big\{ u \in H^1(\R^N)\setminus X^- : \ \cJ'(u)|_{\R u + X^-} = 0 \big\},
$$
which contains all nontrivial critical points of $\cJ$.

Our main results read as follows.

\begin{Th}\label{Th:Existence}
Assume that (G1)--(G5) hold. Then \eqref{eq} has a nontrivial solution $u \in H^1 (\R^N)$ such that $0<\inf_{\cN}\cJ\leq\cJ(u)\leq c$, where $c$ is the minimax level given by \eqref{cm} or \eqref{c}.
\end{Th}

We note also that in the case $K = \R^N$, we obtain $c=\inf_{\cN}\J$ and $u$ is a ground state solution, i.e. $u$ minimizes the energy on $\cN$, cf. \eqref{eq:mptlevel}. Hence, $u$ is the least energy solution and we recover results from \cite{Liu,Mederski2016} and also from \cite{SzulkinWeth}, where the stronger monotonicity condition has been assumed.

\begin{Th}\label{ThMain2}
Assume that (G1)--(G5) hold and $g$ is odd in $u$. Then there exists infinitely many pairs $\pm u_n$ of geometrically distinct solutions to \eqref{eq}, i.e. solutions such that $(\Z^N\ast u_n)\cap (\Z^N\ast u_m) =\emptyset$ for $n\neq m$, where
$\Z^N\ast u_n:=\{u_n(\cdot+z):\; z\in\Z^N\}.$
\end{Th}

Under our conditions we show that the energy functional $\cJ$ has the linking geometry and Cerami sequences are bounded. This allows to use a variant of linking theorem obtained in \cite{Mederski2016,LiSzulkin}, to obtain one solution. 
However, the multiplicity of solutions seem to be difficult to obtain by standard methods, e.g. by means of techniques demonstrated in \cite{CotiZelatiRabibowitz,KryszSzulkin,SzulkinWeth,dePaivaKryszewskiSzulkin}. Even if $K=\R^N$ and (G1)--(G4) are satisfied, we present a new approach for \eqref{eq} in comparison to previous works, see \cite{SzulkinWeth,Mederski2016,dePaivaKryszewskiSzulkin,Liu} and references therein. Namely,
in order to obtain the existence of one solution below the minimax level \eqref{cm} as well as infinitely many solutions, we observe that for any $u\in H^1(\R^N)$ the functional $\cJ(u+\cdot)$ is strictly concave on $X^-$, which allows us to reduce the problem to looking for critical points of a reduced functional $\tJ$ on $X^+$. Although the nonlinear term $G$ is not super-quadratic at infinity on the whole space $\R^N$, we can show that $\tJ$ has the mountain pass geometry. It is crucial to construct an infinite dimensional subspace $\cQ \subset X^+$ on which our nonlinear term $\int_{\R^N}G(x,u)\,dx$  is super-quadratic at infinity. In the multiplicity result we use a Cerami-type condition (see $(M)_\alpha^\beta$ in Section \ref{sec:criticaslpoitth}) and a variant of Benci's pseudoindex based on the Krasnoelskii genus. In fact we refine a recent critical point theory from  \cite[Section 3]{MederskiSchinoSzulkin} for strongly indefinite functionals, which do not have to be globally super-quadratic  -- see Section \ref{sec:criticaslpoitth} for details. Note that in our approach in case $K=\R^N$ we do not have to involve a topological generalized Nehari manifold \cite{SzulkinWeth} or set-valued nonsmooth analysis \cite{dePaivaKryszewskiSzulkin}. Moreover we are able to treat a wider range of problems with $K\subset \R^N$.

The paper is organized as follows. The second section consists of an abstract setting which allows us to find a Cerami sequence for $\cJ$ and to show the multiplicity of solutions. In Section \ref{sect:Q} we show our construction of an appropriate subspace $\cQ \subset X^+$ which consists of functions being zero outside of $K$. In  Section \ref{sec:Assumpt} we verify assumptions from the abstract setting and in Section \ref{sec:Cerami} we show the boundedness of Cerami-type sequences. Section \ref{sect:existence}  contains proofs of Theorems \ref{Th:Existence} and \ref{ThMain2} respectively. In Appendix \ref{sect:appA} we provide also another proof of the existence of a nontrivial solution by means of a linking-type theorem.

\section{Critical point theory}\label{sec:criticaslpoitth}

The following abstract setting is recalled from \cite{BartschMederski,BartschMederskiJFA,MederskiSchinoSzulkin}, where  super-quadratic problems have been considered. Our aim is to refine this theory for partially super-quadratic problems as \eqref{eq}.
Let $X$ be a reflexive Banach space with the norm $\|\cdot\|$ and a topological direct sum decomposition $X=X^+\oplus X^-$, where $X^+$ is a Hilbert space with a scalar product $\langle \cdot , \cdot \rangle$. For $u\in X$ we denote by $u^+\in X^+$ and $u^- \in X^-$ the corresponding summands so that $u = u^++u^-$. We may assume $\langle u,u \rangle = \|u\|^2$ for any
$u\in X^+$ and $\|u\|^2 = \|u^+\|^2+\|u^- \|^2$. We introduce the topology $\cT$ on $X$ as the product of the norm topology in $X^+$ and the weak topology in $X^-$. Hence $u_n\cTto u$ if and only if $u_n^+\to u^+$ and $u_n^-\weakto u^-$.

Let $\cJ:X\to\R$ be a functional of the form 
\begin{equation}\label{EqJ}
\cJ(u) := \frac12\|u^+\|^2-\cI(u) \quad\text{for $u=u^++u^- \in X^+\oplus X^-$}
\end{equation}
and let us define the set
\begin{equation}\label{eq:ConstraintM}
\cM := \{u\in X:\, \cJ'(u)|_{X^-}=0\}=\{u\in X:\, \cI'(u)|_{X^-}=0\}.
\end{equation}
Clearly $\cM$ contains all critical points of $\cJ$ and we assume the following conditions introduced in \cite{BartschMederski,BartschMederskiJFA}:
\begin{itemize}
	\item[(I1)] $\cI\in\cC^1(X,\R)$ and $\cI(u)\ge \cI(0)=0$ for any $u\in X$.
	\item[(I2)] $\cI$ is $\cT$-sequentially lower semicontinuous:
	$u_n\cTto u\quad\Longrightarrow\quad \liminf \cI(u_n)\ge \cI(u)$.
	\item[(I3)] If $u_n\cTto u$ and $\cI(u_n)\to \cI(u)$ then $u_n\to u$.
	\item[(I4)] $\|u^+\|+\cI(u)\to\infty$ as $\|u\|\to\infty$.
	\item[(I5)] If $u\in\cM$ then $\cI(u)<\cI(u+v)$ for every $v\in X^- \setminus\{0\}$.
\end{itemize}
Observe that if $\cI$ is strictly convex and satisfies (I4), then (I2) and (I5) clearly hold. Moreover, for any $u\in X^+$ we find  $m(u)\in\cM$ which is the unique global maximizer of $\J|_{u+X^-}$. Note that $m$ needs not be $\cC^1$, and $\cM$ needs not be a differentiable manifold because $\cI'$ is only required to be continuous.
The following properties (i)--(iv) are taken from \cite[Proof of Theorem~4.4]{BartschMederskiJFA} and we observe that they are implied by (I1)--(I5).
\begin{itemize}
	\item[(i)] For each $u^+\in X^+$ there exists a unique $u^-\in X^-$ such that $m(u^+):=u^++u^-\in\cM$. This $m(u^+)$ is the minimizer of $\cI$ on $u^++X^-$. 
	\item[(ii)] $m:X^+\to \cM$  is a homeomorphism with the inverse $\cM\ni u\mapsto u^+\in X^+$. 
	\item[(iii)] $\wt{\cJ}:=\cJ\circ m \in\cC^1(X^+,\R)$.
	\item[(iv)]$\wt{\cJ}'(u^+) = \cJ'(m(u^+))|_{X^+}:X^+\to\R$ for every $u^+\in X^+$.
\end{itemize}
In order to get the mountain pass geometry of $\wt{\cJ}$ we need some additional assumptions.
\begin{itemize}
	\item[(I6)] There exists $r>0$ such that $a:=\inf\limits_{u\in X^+,\|u\|=r} \cJ(u)>0$.
	\item[(I7)] There is an infinite dimensional closed vector subspace $\cQ \subset X^+$ such that $\cI(m(t_nu_n))/t_n^2\to\infty$ if $t_n\to\infty$, $u_n \in \cQ$ and $u_n\to u\ne 0$ as $n\to\infty$.
\end{itemize}

Note that in the previous works \cite{BartschMederskiJFA,MederskiSchinoSzulkin}, instead of (I7), the following stronger condition  has been assumed:
\begin{equation}\label{eq:I7strong}
\cI(t_nu_n)/t_n^2\to\infty\hbox{ if }t_n\to\infty,\; u_n \in X\hbox{ and }u_n^+\to u^+\ne 0\hbox{ as }n\to\infty.
\end{equation}
Since our problem \eqref{eq} is not super quadratic outside $K$, we require only (I7) and we introduce $\cQ$ containing functions with support in $K$.\\ 
\indent Recall from \cite{BartschMederskiJFA} that $(u_n)$ is called a $(PS)_c$-sequence for $\cJ$ if $\cJ'(u_n)\to 0$ and $\cJ(u_n)\to c$, and $\cJ$ satisfies the $(PS)_c^\cT$-condition on $\cM$ if each $(PS)_c$-sequence $(u_n)\subset\cM$ has a subsequence converging in the $\cT$-topology. Since we look for solutions to \eqref{eq} in $\R^N$ and not in a bounded domain as in \cite{BartschMederskiJFA}, the $(PS)_{c}^\cT$-condition is no longer satisfied.\\ 
\indent Now our approach is similar to \cite{MederskiSchinoSzulkin}, but we need to work with the weaker condition than \eqref{eq:I7strong}.
Note that by (I5) and (I6),
$\tJ(u)\geq \J(u)\geq a$ for $u\in X^+$ and $\|u\|=r$, and  $\tJ(tu)/t^2\to -\infty$ as $t\to\infty$ for  $u\in \cQ$. Therefore $\tJ$ has the mountain pass geometry and similarly as in \cite[Theorem 4.4]{BartschMederskiJFA} we may define the mountain pass level
 \begin{equation}\label{cm}
 c_\cM := \inf_{\ga\in\Ga}\sup_{t\in [0,1]} \cJ(\ga(t)),
 \end{equation}
 where
 \[
 \Ga := \{\ga\in\cC([0,1],\cM) \ : \ \ga(0)=0, \ \|\gamma(1)^+\|>r,\text{ and } \cJ(\ga(1))<0\}.
 \]
In order to show that $c_{\cM}\ge a>0$ we require the following condition on $\cI$:
 \begin{itemize}
 	\item[(I8)] $\frac{t^2-1}{2}\cI'(u)[u]+\cI(u) - \cI(tu+v)=\frac{t^2-1}{2}\cI'(u)[u] + t\cI'(u)[v] + \cI(u) - \cI(tu+v) \leq 0$\\ for every $u\in \cN$, $t\ge 0$, $v\in X^-$,
 \end{itemize}
 where 
 \begin{equation}\label{eq:NehariDef}
 \cN := \{u\in X\setminus X^-: \cJ'(u)|_{\R u+ X^-}=0\} = \{u\in\cM\setminus X^-: \cJ'(u)[u]=0\} \subset\cM.
 \end{equation}

\begin{Th}\label{ThLink1}
	Suppose $\cJ \in \cC^1(X,\R)$ satisfies (I1)--(I8). Then $\tJ$ has a Cerami sequence $(u_n)$ at the level $c_\cM$ and
 \begin{equation}\label{eq:mptlevel}
0<a\leq c_\cM\leq \inf_{\ga\in\Ga,\, \gamma([0,1])\subset m(\cQ)}\sup_{t\in [0,1]} \cJ(\ga(t))=\inf_{\cN\cap m(\cQ)}\cJ=\inf_{u\in\cQ}\sup_{t\geq 0} \J(m(tu)).
 \end{equation}
\end{Th}
\begin{proof}
Observe that for every $u\in \cQ\setminus\{0\}$,  the map $(0,+\infty)\ni t\mapsto \tJ(tu)\in\R$ attains maximum at some point $t_0>0$ and $\tJ'(t_0u)(u)=0$. Hence $m(t_0u)\in\cN$ and $\cN\cap m(\cQ)\neq\emptyset$.
Note that (I8) is equivalent to
\begin{equation}\label{eq:JonN}
\cJ(tu+v)-\cJ'(u)\Big[\frac{t^2-1}{2}u+tv\Big]\leq \cJ(u)
\end{equation}
for any $u\in \cN$, $v\in X^-$ and $t\geq 0$. 
Hence, if $u\in X^+$, $m(t_0u),m(t_1u)\in \cN$ for some $t_0,t_1>0$, then by \eqref{eq:JonN}, $\wt{\cJ}(t_1u) =\wt{\cJ}(t_2u)$. Suppose that $u\in m^{-1}(\cN)$.
Then there exist $0 < t_{min} \le 1 \leq  t_{max}$ such that $m(tu)\in\cN$ if and only if $t\in[t_{min},t_{max}]$ and $\wt{\cJ}(tu)$ has the same value for those $t$. Hence $\wt{\cJ}'(tu)[u]>0$ for $0<t<t_{min}$ and $\wt{\cJ}'(tu)[u]<0$ for $t>t_{max}$. It follows that $\cQ\setminus m^{-1}(\cN)$ consists of two connected components and any $\sigma\in \wt{\Gamma}$ intersects $m^{-1}(\cN)$, where
\begin{equation}\label{eq:Gamma_Phi}
\wt{\Gamma} := \{\sigma\in\cC([0,1],\cQ) \ : \ \sigma(0)=0, \ \|\sigma(1)\|>r \hbox{ and }\tJ(\sigma(1))<0\}.
\end{equation}
Thus
$$\inf_{\sigma\in\wt{\Gamma}}\sup_{t\in [0,1]}\cJ\circ m(\sigma(t))\geq \inf_{\cN\cap m(\cQ)}\J.$$
Note that
$$c_{\cM}\leq \inf_{\sigma\in\wt{\Gamma}}\sup_{t\in [0,1]}\cJ\circ m(\sigma(t))\leq \inf_{\cN\cap m(\cQ)}\J = \inf_{u\in \cQ\setminus\{0\}}\max_{t>0}\tJ(tu)$$
and we conclude \eqref{eq:mptlevel}.
By the mountain pass theorem there exists a Cerami sequence $(u_n)$ for $\tJ$ at the level $c_\cM\geq a$ (see \cite{Cerami,bbf}).
\end{proof}

In order to deal with multiplicity of critical point, we introduce a discrete group action on $X$, e.g. in our application to \eqref{eq} we have $G=\mathbb{Z}^N$ acting by translations, see Theorem \ref{ThMain2}.\\
\indent For a topological group acting on $X$, denote \emph{the orbit of $u\in X$} by $G\ast u$, i.e., 
$$G\ast u:=\{gu \ : \ g\in G\}.$$
A set $A\subset X$ is called \emph{$G$-invariant} if $gA\subset A$ for all $g\in G$. $\cJ: X\to\R$ is called \emph{$G$-invariant} and $T: X\to X^*$ \emph{$G$-equivariant} if $\cJ(gu)=\cJ(u)$ and $T(gu)=gT(u)$ for all $g\in G$, $u\in X$. \\
\indent   
In order to deal with  multiplicity of critical points, assume that $G$ is a topological group such that
\begin{itemize}
	\item[(G)] $G$ acts on $X$ by isometries and discretely in the sense that for each $u\ne 0$, $(G*u)\setminus\{u\}$ is bounded away from $u$. Moreover, $\cJ$ is $G$-invariant and $X^+, X^-$ are $G$-invariant.
\end{itemize}
Observe that $\cM$ is $G$-invariant and $m:X^+\to\cM$ is $G$-equivariant.  

\begin{Lem}[\cite{MederskiSchinoSzulkin}] \label{discrete}
	For all $u,v\in X$ there exists $\eps=\eps_{u,v}>0$ such that $\|gu-hv\|>\eps$ unless $gu=hv$, where $g,h\in G$.
\end{Lem}

We shall use the notation 
\begin{gather*}
\tJ^\beta := \{u\in X^+ \ : \ \tJ(u)\le\beta\}, \quad  \tJ_\alpha := \{u\in X^+\ : \  \tJ(u)\ge\alpha\}, \\
\tJ_\alpha^\beta := \tJ_\alpha\cap\tJ^\beta,\quad \cK:=\big\{u\in X^+\ : \ \tJ'(u)=0\big\}.
\end{gather*}
Note that by \eqref{eq:JonN}
$$\cJ(u)\geq \cJ\left(\frac{r}{\|u^+\|}u^+\right)\geq a$$
for any $u\in\cN$, hence $\inf_{\cN}\J\geq a>0$.
Since all nontrivial critical points of $\cJ$ are in $\cN$, $\tJ(u)\ge a$ for all $u\in\cK\setminus\{0\}$.

We recall the following variant of the {\em Cerami condition} between the levels $\alpha, \beta\in\R$ introduced in \cite{MederskiSchinoSzulkin}.
\begin{itemize}
	\item[$(M)_\alpha^\beta$]
	\begin{itemize}
		\item[(a)] Let $\alpha\le\beta$. There exists $M_\alpha^\beta$ such that $\limsup_{n\to\infty}\|u_n\|\le M_\alpha^{\beta}$ for every $(u_n)\subset X^+$ satisfying $\alpha\le\liminf_{n\to\infty}\tJ(u_n)\le\limsup_{n\to\infty}\tJ(u_n)\leq\beta$ and \linebreak $(1+\|u_n\|)\tJ'(u_n)\to 0$.
		\item[(b)] Suppose in addition that the number of critical orbits in $\tJ_\alpha^\beta$ is finite. Then there exists $m_\alpha^\beta>0$ such that if $(u_n), (v_n)$ are two sequences as above and $\|u_n-v_n\|<m_\alpha^\beta$ for all $n$ large, then  $\liminf_{n\to\infty}\|u_n-v_n\|=0$.
	\end{itemize}
\end{itemize}

Note that if $\cJ$ is even, then $m$ is odd (hence $\tJ$ is even) and $\cM$ is symmetric,  i.e. $\cM=-\cM$. Note also that $(M)_\alpha^\beta$ is a condition on $\tJ$ and \emph{not} on $\cJ$. Our main multiplicity result reads as follows.
\begin{Th}\label{Th:CrticMulti}
	Suppose $\cJ \in \cC^1(X,\R)$ satisfies (I1)--(I8) $\cJ$ is even. If $(M)_0^{\beta}$ holds for every $\beta>0$,
	then $\cJ$ has infinitely many distinct critical orbits.
\end{Th}

If $\cQ=X^+$, 
the above result has been obtained in \cite[Theorem 3.5 (b)]{MederskiSchinoSzulkin} and proof of Theorem \ref{Th:CrticMulti} is similar. For the reader's convenience we recall some important steps and we prove results, where $\cQ\subset X^+$ and (I7) play an important role.

\begin{Lem} \label{infty}
	Let $\cQ_k$ be a $k$-dimensional subspace of $\cQ$. Then $\tJ(u)\to-\infty$ whenever $\|u\|\to\infty$ and $u\in \cQ_k$.
\end{Lem}

\begin{proof}
	It suffices to show that each sequence $(u_n)\subset \cQ_k$ such that $\|u_n\|\to\infty$ contains a subsequence along which $\tJ(u_n)\to-\infty$.  Let $u_n = t_nv_n$,  $\|v_n\|=1$.
Then, passing to a subsequence, $v_n \to v_0$, $v_0 \in \cQ$ and $\|v_0\|=1$. Hence by (I7)
$$
\frac{\tJ(u_n)}{t_n^2}\leq \frac12 - \frac{\cI(m(t_nv_n))}{t_n^2}\to-\infty
$$
as $n\to\infty$.
\end{proof}

As usual, $(u_n)\subset X^+$ will be called \emph{a Cerami sequence} for $\tJ$ at the level $c$ if $(1+\|u_n\|)\tJ'(u_n) \to 0$ and $\tJ(u_n) \to c$. In view of (I4), it is clear that if $(u_n)$ is a bounded Cerami sequence for $\tJ$, then $(m(u_n))\subset \cM$ is a bounded Cerami sequence for $\cJ$.

By a standard argument we can find
a locally Lipschitz continuous pseudo-gradient vector field $v:X^+\setminus \cK\to X^+$ associated with $\wt{\cJ}$, i.e.
\begin{eqnarray}
\|v(u)\|&<&1,\label{eq:flow2}\\
\wt{\cJ}'(u)[v(u)] &>& \frac12\|\wt{\cJ}'(u)\|\label{eq:flow3}
\end{eqnarray}
for any $u\in X^+\setminus \cK$. Moreover, if $\J$ is even, then $v$ is odd.  
Let
$\eta:\cG\to  X^+\setminus \cK$ be the flow defined by
\begin{equation*}
\left\{
\begin{aligned}
&\partial_t \eta(t,u)=-v(\eta(t,u))\\
&\eta(0,u)=u
\end{aligned}
\right.
\end{equation*}
where $\cG:=\{(t,u)\in [0,\infty)\times (X^+\setminus \cK) \ : \ t<T(u)\}$ and $T(u)$ is the maximal time of existence of $\eta(\cdot,u)$.  We prove Theorem \ref{Th:CrticMulti} by contradiction and
from now on we assume that there is a finite number of distinct orbits $\{G\ast u: \ u\in \cK\}$. Recall the following lemma.

\begin{Lem}[\cite{MederskiSchinoSzulkin}]\label{lem:flow}
	Suppose $(M)_0^\beta$ holds for some $\beta>0$ and let $u \in \tJ_0^\beta\setminus \cK$. Then either $\lim_{t \to T(u)}\eta(t,u)$ exists and is a critical point of $\wt{\cJ}$ or $\lim_{t\to T(u)}\wt{\cJ}(\eta(t,u)) = -\infty$. In the latter case $T(u)=\infty$.
\end{Lem}

Similarly as in \cite{MederskiSchinoSzulkin}, let
$\Sigma := \{A\subset X^+ \ : \ A=-A \text{ and } A \text{ is compact}\}$,
\[
\cH := \{h \ : \ X^+\to X^+ \text{ is a homeomorphism, }  h(-u)=-h(u) \text{ and } \wt{\cJ}(h(u))\le \wt{\cJ}(u) \text{ for all } u\},
\]
and for $A\in\Sigma$ we put
\[
i^*(A) := \min_{h\in\cH} \gamma(h(A)\cap S(0,r)),
\]
where $S(0,r):=\{u\in X^+ \ : \ \|u\|=r\}$ and $\gamma$ is Krasnoselskii's genus \cite{Struwe}. This is a variant of Benci's  pseudoindex \cite{bbf, be} and the following properties are adapted from \cite[Lemma 2.16]{SquassinaSzulkin}.

\begin{Lem} \label{index}
	Let $A,B\in\Sigma$.\\
	(i) If $A\subset B$, then $i^*(A)\le i^*(B)$. \\
	(ii) $i^*(A\cup B)\le i^*(A)+\gamma(B)$. \\
	(iii) If $g\in \cH$, then $i^*(A)\le i^*(g(A))$. \\
	(iv) Let $\cQ_k$ be a $k$-dimensional subspace of $\cQ$ given in (I7). Then $i^*(D_k)\ge k$, where $D_k:=\cQ_k\cap \overline B(0,R)$ and $R$ is large enough.

\end{Lem}

\begin{proof}
	(i)--(iii) are proved in  \cite[Lemma 3.7]{MederskiSchinoSzulkin}.\\
\indent (iv) By Lemma \ref{infty}, $\tJ(u)<0$ on $\cQ_k\setminus B(0,R)$ if $R$ is large enough. Let $D_k := \cQ_k\cap\overline B(0,R) \neq \{0\}$ and note that $D_k \subset X^+$ is compact and symmetric, i.e. $D_k \in \Sigma$. Suppose $i^*(D_k)<k$, choose $h\in\cH$ such that $\gamma(h(D_k)\cap S(0,r))<k$ and an odd mapping 
	$$f: h(D_k)\cap S(0,r) \to \R^{k-1}\setminus\{0\}.$$ 
	Let $U:= h^{-1} (B(0,r)) \cap \cQ_k$. 
	Observe that $\wt{\cJ}(h(u))\le \wt{\cJ}(u)< 0$ for $u\in \cQ_k\setminus B(0,R)$ and $\wt{\cJ}(u)\ge 0$ for $u\in B(0,r)$. Suppose that there is $u \in U$ such that $u \in \cQ_k\setminus B(0,R)$. Since $h(u) \in B(0,r)$ we have
	$$
	0 \leq \wt{\cJ}(h(u)) \leq \wt{\cJ}(u)< 0,
	$$
	which is a contradiction. Hence $U\subset D_k$. Since $h$ is a homeomorphism we see that $U$ is open in $\cQ_k$. Since $U \subset D_k$, we see that $U$ is bounded and $0 \in U$. Therefore $U$ is bounded, open neighbourhood of $0$ in $\cQ_k$. If $u\in \partial U$, then $h(u)\in S(0,r)$ and therefore $f\circ h: \partial U \to \R^{k-1}\setminus\{0\}$, contradicting the Borsuk-Ulam theorem \cite[Proposition II.5.2]{Struwe}, \cite[Theorem D.17]{Willem}. Therefore $i^*(D_k)\ge k$.

\end{proof}

\begin{altproof}{Theorem \ref{Th:CrticMulti}}
	Take $\beta\geq a$ and let
	\begin{align*}
		\cK^\beta := \{u\in \cK \ : \  \wt{\cJ}(u)=\beta\}.
	\end{align*}
	Since there are finitely many critical orbits, there exists $\eps_0>0$ for which 
	\begin{equation}\label{eq:KL}
	\cK\cap \wt{\cJ}_{\beta-\eps_0}^{\beta+\eps_0}=\cK^\beta.
	\end{equation} 
	Choose $\delta\in (0,m_0^{\beta+\eps_0})$ such that $\overline B(u,\delta)\cap \overline B(v,\delta) = \emptyset$ for all $u,v\in\cK^\beta$, $u\ne v$ (this is possible due to Lemma \ref{discrete}). Similarly as in \cite{MederskiSchinoSzulkin} we show there is $\eps\in (0,\eps_0)$ such that
	\begin{equation}\label{eq:entrancetime1}
	\lim_{t\to T(u)} \wt{\cJ}(\eta(t,u)) < \beta -\eps \quad\hbox{for } u\in \wt{\cJ}^{\beta+\eps}_{\beta-\eps}\setminus B(\cK^\beta,\delta).
	\end{equation}
	Define 
	\[
	\beta_k :=  \inf_{i^*(A)\ge k} \sup_{u\in A}\wt{\cJ}(u), \quad k=1,2,\ldots.
	\]
	and note that by Lemma \ref{index}  all $\beta_k$ are well defined, finite and $a\le\beta_1\le\beta_2\le\ldots$.
	Let $\beta=\beta_k$  for some $k\geq 1$ and
	take $\eps>0$ such that \eqref{eq:entrancetime1} holds. 
	As in \cite{MederskiSchinoSzulkin} we define the flow $\wt\eta:\R\times X^+\to  X^+$ such that $\wt\eta(t,u)=\eta(t,u)$ as long as $t\ge 0$ and $ \wt\eta(t,u)\in \wt{\cJ}^{\beta+\eps}_{\beta-\eps}\setminus B(\cK^\beta,\delta/2)$. Now, using \eqref{eq:entrancetime1} we can define the entrance time map $e:\wt{\cJ}^{\beta+\eps}\setminus B(\cK^\beta,\delta)\to [0,\infty)$ by
	$$e(u):=\inf\{t\in [0,\infty) \ : \  \wt{\cJ}(\wt\eta(s,u))\leq \beta -\eps\}.$$ 
	Then $e(u)$ is finite and it is standard to show that $e$ is continuous and even. Take any $A\in\Sigma$ such that $i^*(A)\geq k$ and $\wt{\cJ}(u)\leq \beta+\eps$ for $u\in A$. Let $T:=\sup_{u\in A} e(u)$ and set $h:=\wt\eta(T,\cdot)$. Observe that $h\in\cH$,
	$$i^*(A\setminus B(\cK^\beta,\delta))\leq 
	i^*(h(A\setminus B(\cK^\beta,\delta)))\leq  k-1$$
	and
	\begin{equation}\label{eq:LSvaluse}
	k\leq i^*(A)\leq \gamma (\overline B(\cK^\beta,\delta)\cap A)+
	i^*(A\setminus B(\cK^\beta,\delta))
	\leq \gamma(\cK^\beta)+k-1.
	\end{equation}
	Thus $\cK^\beta\neq \emptyset$ and $\cK^\beta$ is (at most) countable, so that 
	$$\gamma (\overline B(\cK^\beta,\delta))=\gamma (\cK^\beta)=1.$$
	If $\beta_k=\beta_{k+1}$ for some
	$k\geq 1$, then \eqref{eq:LSvaluse} implies $\gamma(\cK^{\beta_k})\geq 2$, which is a contradiction. Therefore we get an infinite sequence
	$\beta_1<\beta_2<...$
	of critical values which contradicts our assumption that $\cK$ consists of a finite number of distinct orbits. 
	This completes the proof.	
\end{altproof}

\section{Variational setting and construction of $\cQ$}\label{sect:Q}

In view of (V), the Schr\"odinger operator $\cA := -\Delta + V(x) : \cD(\cA) \rightarrow L^2 (\R^N)$ in $L^2 (\R^N)$ is self-adjoint and its domain is $\cD(\cA) = H^2 (\R^N) \subset L^2 (\R^N)$. We set
$$
X := H^1 (\R^N)
$$
with the orthogonal splitting $X = X^+ \oplus X^-$. 
On $X$ we consider the norm given by
$$
\|u\|^2 := \int_{\R^N} |\nabla u^+|^2 + V(x) |u^+|^2 \, dx - \int_{\R^N} |\nabla u^-|^2 + V(x) |u^-|^2 \, dx = \|u^+\|^2 + \|u^-\|^2
$$
and the corresponding scalar product
$$
\langle u, v\rangle := \int_{\R^N} \nabla u^+  \nabla v^+ + V(x) u^+ v^+ \, dx - \int_{\R^N} \nabla u^-  \nabla v^- + V(x) u^- v^- \, dx,
$$
where $u = u^+ + u^- \in X^+ \oplus X^-$. Moreover we can rewrite $\cJ$ in the following form
$$
\cJ (u) = \frac12 \|u^+\|^2 - \frac12 \|u^-\|^2 - \int_{\R^N} G(x,u) \, dx
$$
for $u = u^+ + u^- \in X^+ \oplus X^-$. Then $\mathcal{J} \in \cC^1 (X)$ and critical points of $\cJ$ are weak solutions to \eqref{eq}. If in addition $\inf \sigma (-\Delta + V(x)) > 0$, we have $X^- = \{0\}$ and $X^+ = H^1 (\R^N)$. Otherwise $X^-$ is an infinite dimensional subspace of $X$.

Take any open subset $\Om$ such that $\overline{\Om} \subset \mathrm{int}\, K\cap (0,1)^N$. Since the operator $-\Delta+V(x)$ on $H^{1}_0(\Om)$ has a discrete and unbounded from above spectrum, we define an infinite dimensional  subspace $\cQ$ of $ H^{1}_0(\Om)$ such that $-\Delta+V(x)$ is positive definite on $\cQ$. Clearly $\cQ\subset X^+$ and $\supp (u)\subset  \overline{\Om} $ for $u\in\cQ$. Observe that, if $u_n \to u$ in $X^+$ and $u_n \in \cQ$, we have $u = 0$ a.e. on $\R^N \setminus \overline{\Om} $. Thus, taking $\overline{\cQ}$ instead of $\cQ$ we may assume that $\cQ$ is closed and $\supp(u)\subset K$ for every $u \in \cQ$. We observe the following crucial property of  $\cQ$.

\begin{Lem} \label{lem:Q}
 If $u\in X \setminus X^-$  is such that $u^+\in \cQ$, then 
 $|\supp(u)\cap K|>0$. 
\end{Lem}

\begin{proof}
Assume by contradiction that $u = u^+ + u^- \in X \setminus X^-$ is such that $u^+ \in \cQ$ and $u = 0$ a.e. on $K$. Then $u^+ = - u^-$ a.e. on $K$ and since $u^+ \in \cQ$ we have $u^+ = 0$ a.e. on $\R^N \setminus K$. Moreover $\nabla u^+ = 0$ a.e. on $\R^N \setminus K$ and we obtain that
\begin{align*}
\|u^+\|^2 &= \int_K |\nabla u^+|^2 + V(x) |u^+|^2 \, dx = \int_K |\nabla u^-|^2 + V(x) |u^-|^2 \, dx \\
&= - \|u^-\|^2 - \int_{\R^N \setminus K} |\nabla u^-|^2 + V(x) |u^-|^2 \, dx \\
&= - \|u^-\|^2 - \int_{\R^N \setminus K} |\nabla u|^2 + V(x) |u|^2 \, dx \\
&= - \|u^-\|^2 - \int_{\R^N} |\nabla u|^2 + V(x) |u|^2 \, dx = - \|u^-\|^2 - (\|u^+\|^2-\|u^-\|^2) = - \|u^+\|^2.
\end{align*}
Hence $\|u^+\| = 0$, $u = u^- \in X^-$ and we get a contradiction. Therefore $|\{x\in K: u(x)\neq 0\}|>0$, which implies that
$|\supp(u)\cap K|>0$. 
\end{proof}

\section{Verification of (I1)--(I8)}\label{sec:Assumpt}

Define $\cI (u) := \frac12 \|u^-\|^2 + \int_{\R^N} G(x,u) \, dx$ for $u \in X$.  Note that $\cJ(u) = \frac12 \|u^+\|^2 - \cI(u)$ is of the form \eqref{EqJ}.
Then in view of (G1) and (G2), for any $\eps>0$ we find $c_\eps>0$ such that
\begin{equation}\label{eq:gestimates}
|g(x,u)|\leq \eps |u| + c_\eps |u|^{p-1}.
\end{equation}
Hence $\cJ$ is of $\cC^1$ class and by direct computation we obtain $\cI(0)=0$ and (I1) holds. 

\begin{Lem}\label{lem:convex}
$\cI$ is convex and $\cI(u+\cdot)$ is strictly convex on $X^-$ for every $u\in X$.
\end{Lem}

\begin{proof}
For any $\eps > 0$ define $G_\varepsilon(u) = G(x,u) + \frac{\varepsilon}{p} |u|^p$. Then
$$
\frac{G_\eps(x,u)}{u^2} \to \infty \quad \mathrm{as} \ |u|\to\infty
$$
uniformly in $x\in\R^N$. In view of \cite[Lemma 2.2]{SzulkinWeth}, cf.
 \cite[Remark 3.3(a)]{Mederski2016} we show that
$$
g_\eps(x,u) \left( \frac{t^2-1}{2} u+tv \right)+G_\eps(x,u)-G_\eps(x,tu+v) \leq 0
$$
holds for any $u,v \in \R$, $t \geq 0$, $\eps > 0$ and a.e. $x \in \R^N$, where $g_\varepsilon (x,u) = g(x,u) + \varepsilon |u|^{p-2}u$. For $t = 1$ we get
$$
g_\eps(x,u)v +G_\eps(x,u)-G_\eps(x,u+v) \leq 0
$$
and passing to the limit as $\eps \to 0^+$
$$
g(x,u)v +G(x,u)-G(x,u+v) \leq 0
$$
or equivalently
$$
G(x,u+v) \geq G(x,u) + g(x,u)v.
$$
Thus $G(x, \cdot)$ is convex and therefore
$$
u \mapsto \int_{\R^N} G(x,u) \, dx 
$$
is convex. Since 
$$
u^- \mapsto \frac12 \|u^-\|^2 
$$
is strictly convex on $X^-$, we see that $\cI(u+\cdot):X^-\to\R$ is also strictly convex.
\end{proof}

Clearly, since $\cI$ is convex, (I2) is satisfied. Now we show (I4). Take any sequence $(u_n) \subset X$ and suppose that $\|u_n\| \to \infty$. If $\|u_n^+\| \to \infty$ we see that
$$
\|u_n^+\| + \cI(u_n) \geq \|u_n^+\| \to \infty.
$$
Otherwise, $(u_n^+)$ is bounded and $\|u_n^-\| \to \infty$. Hence
$$
\|u_n^+\| + \cI(u_n) \geq \cI(u_n) = \frac12 \|u_n^-\|^2 + \int_{\R^N} G(x,u_n) \, dx \geq \frac12 \|u_n^-\|^2 \to \infty.
$$
Now by Lemma \ref{lem:convex} and (I4), we easy check that $\cM$ is nonempty and (I5) is satisfied. Suppose that $u_n\cTto u$, i.e. $u_n^+\to u^+$ and $u_n^- \weakto u^-$. Observe that passing to a subsequence
$$
\liminf_{n\to\infty}\Big(\frac12 \|u_n^-\|^2 + \int_{\R^N} G(x,u_n) \, dx \Big)\geq \frac12 \|u^-\|^2 + \int_{\R^N} G(x,u) \, dx
$$
and if, in addition, $\cI(u_n) \to \cI(u)$, we obtain that
$$
\|u_n^-\|^2 \to \|u^-\|^2 \mbox{ and } \int_{\R^N} G(x,u_n) \, dx \to \int_{\R^N} G(x,u) \, dx.
$$
Thus $u_n^- \to u^-$ and (I3) holds. Note that \eqref{eq:gestimates} implies (I6). Hence we only need to check (I7) and (I8).

\begin{Lem}
(I7) holds.
\end{Lem}

\begin{proof}
Since $u_n \in \cQ$ we also have $t_n u_n \in \cQ$. Recall that for any $u\in X^+$ we find  $m(u)\in\cM$, which is the unique global maximizer of $\J|_{u+X^-}$ as in Section \ref{sec:criticaslpoitth}. Taking into account that $m(t_n u_n) = t_n u_n + w_n$ for some $w_n \in X^-$, in view of Lemma \ref{lem:Q}, $|\supp (m(t_n u_n)) \cap K | > 0$. Put $v_n:= w_n / t_n$. Note that if $\|v_n\| \to \infty$, then
\begin{align*}
\frac{\cI(m(t_n u_n))}{t_n^2} = \|v_n\|^2 + \int_{\R^N} \frac{G(x,m(t_n u_n))}{t_n^2} \, dx \geq \|v_n\|^2 \to \infty.
\end{align*}
Hence we may assume that $(v_n)$ is bounded, $v_n \weakto v$ and $v_n(x) \to v(x)$ for a.e. $x \in \R^N$. Since $u_n \to u \neq 0$ we may also assume that $u_n(x) \to u(x)$. $\cQ$ is closed, so that $u \in \cQ$ and, again in view of Lemma \ref{lem:Q} we have $|\supp (u+v) \cap K| > 0$. Then, for a.e. $x \in \supp (u+v) \cap K$
$$
|m(t_n u_n)(x)| = |t_n u_n(x) + w_n(x) | = t_n |u_n(x) + v_n(x)| \to \infty
$$
and
$$
\frac{|m(t_n u_n)(x)|^2}{t_n^2} = \left| u_n (x) + v_n(x) \right|^2 \to |u(x) + v(x)|^2 \neq 0.
$$
Moreover, from Fatou's lemma
\begin{align*}
\frac{\cI(m(t_n u_n))}{t_n^2} &= \|v_n\|^2 + \int_{\R^N} \frac{G(x,m(t_n u_n))}{t_n^2} \, dx \geq \int_{K} \frac{G(x,m(t_n u_n))}{t_n^2} \, dx \\
&\geq \int_{\supp (u+v) \cap K} \frac{G(x,m(t_n u_n))}{|m(t_n u_n))|^2} \frac{|m(t_n u_n))|^2}{t_n^2} \, dx \to \infty
\end{align*}
and we conclude.
\end{proof}

Observe that (I8) is a simple consequence of the following inequality.

\begin{Lem}\label{ineq}
For any $u \in X$, $v \in X^-$ and $t \geq 0$ there holds
$$
\cJ(u) \geq \cJ(tu + v) - \frac{t^2-1}{2} \cJ'(u)(u) - t\cJ'(u)(v).
$$
\end{Lem}

\begin{proof}
Define
$$
\cJ_\varepsilon (u) := \cJ(u) - \frac{\varepsilon}{p} |u|_p^p
$$
for any $\varepsilon > 0$. Here and below $|\cdot|_k$ stands for the usual $L^k$-norm, $k\geq 1$ or $k=\infty$.  Then for every $\varepsilon > 0$
$$
\frac{G_\varepsilon (x,u)}{u^2} \to \infty \quad \mbox{as } |u|\to\infty \ \mbox{uniformly in } x \in \R^N,
$$
where
$$
G_\varepsilon (x,u) = G(x,u) + \frac{\varepsilon}{p} |u|_p^p
$$
and as in  \cite[Lemma 3.2]{Mederski2016} (cf.  \cite[Lemma 2.2]{SzulkinWeth}) we check that for any $u \in X$, $v \in X^-$ and $t \geq 0$ there holds
$$
\cJ_\varepsilon (u) \geq \cJ_\varepsilon(tu + v) - \frac{t^2-1}{2} \cJ_\varepsilon'(u)(u) - t\cJ_\varepsilon '(u)(v).
$$
Equivalently, we obtain
$$
\cJ (u) - \frac{\varepsilon}{p} |u|_p^p \geq \cJ(tu+v) - \frac{\varepsilon}{p} |tu+v|_p^p - \frac{t^2-1}{2} \left( \cJ'(u)(u) - \varepsilon |u|^p_p \right) - t \left( \cJ'(u)(v) - \varepsilon \int_{\R^N} |u|^{p-2}uv \, dx \right).
$$
Taking $\varepsilon \to 0^+$ we obtain
$$
\cJ(u) \geq \cJ(tu + v) - \frac{t^2-1}{2} \cJ'(u)(u) - t\cJ'(u)(v).
$$
\end{proof}

\section{Boundedness of Cerami-type sequences}\label{sec:Cerami}

\begin{Lem}\label{lem:CerBounded}
Let $\beta\geq 0$. Any sequence $(u_n) \subset X$ such that
$$
0\leq \cJ(u_n) \leq \beta,\quad
(1+\|u_n^+\|)\J'(u_n)\to 0\hbox{ and }\J'(u_n)(u_n^-)\to 0\hbox{ as }n\to\infty,
$$
is bounded in $X$. In particular, any Cerami sequence for $\cJ$ at a positive level is bounded.
\end{Lem}

\begin{proof}
Assume by contradiction that $\|u_n\|\to\infty$. Put $v_n := u_n / \|u_n\|$. Since $\|v_n\| = 1$, we may assume that $v_n \weakto v$ and $v_n(x) \to v(x)$ for a.e. $x \in \R^N$, passing to a subsequence if necessary.  Moreover we can assume that there is $(z_n) \subset \Z^N$ such that
$$
\liminf_{n\to\infty} \int_{B(z_n, 1+\sqrt{N})} |v_n^+|^2 \, dx > 0.
$$
Otherwise, in view of Lions' lemma \cite[Lemma 1.21]{Willem}
$$
v_n^+ \to 0 \quad \mbox{in} \ L^t(\R^N)
$$
for all $2 < t < 2^*$. Fix any $s > 0$ and $\varepsilon >0$, in view of (G1) and (G2) there is $C_\varepsilon > 0$ such that
\begin{align*}
\limsup_{n \to \infty} \left| \int_{\R^N} G(x, s v_n^+) \, dx \right| &\leq \limsup_{n\to\infty} \left( \varepsilon  |sv_n^+|_2^2 + C_\varepsilon  |sv_n^+|_p^p \right) \\ 
&\leq \varepsilon \limsup_{n\to\infty} |sv_n^+|_2^2.
\end{align*}
Taking $\varepsilon \to 0^+$ we get 
$$
\int_{\R^N} G(x, s v_n^+) \, dx \to 0
$$
for any $s > 0$. Since $\J'(u_n)(u_n)\to 0$ and $\J'(u_n)(u_n^-)\to 0$ and taking into account Lemma \ref{ineq}, we infer that
\begin{eqnarray*}
\cJ(u_n) &\geq& \cJ(s v_n^+) - \frac{(s/\|u_n\|)^2-1}{2} \cJ'(u_n)(u_n) + (s/\|u_n\|)^2\cJ'(u_n)(u_n^-)\\
&=& \cJ(s v_n^+) + o(1)= \frac{s^2}{2} \|v_n^+\|^2 + o(1).	
\end{eqnarray*}
Note that 
$$\|v_n^+\|^2- \|v_n^-\|^2\geq 2\cJ(u_n)\geq 0,$$
hence 
$$\cJ(u_n) \geq \frac{s^2}{2} \|v_n^+\|^2 + o(1)\geq   \frac{s^2}{4} (\|v_n^+\|^2+\|v_n^-\|^2) + o(1)=\frac{s^2}{4} + o(1)$$
for any $s\geq 0$,
and we get a contradiction, since $(\cJ(u_n))$ is bounded. Thus there is $(z_n) \subset \Z^N$ such that
$$
\liminf_{n\to\infty} \int_{B(z_n, 1+\sqrt{N})} |v_n^+|^2 \, dx > 0.
$$
Passing to a subsequence we have $v_n (\cdot + z_n) \weakto v \neq 0$ and $v_n (x + z_n) \to v(x)$ for a.e. $x \in \R^N$. Suppose that
$$
|S| > 0, \mbox{ where } S = \supp v \cap K.
$$
Note that for a.e. $x \in S$ we have $|u_n(x+z_n)| = |v_n(x+z_n)| \| u_n \| \to \infty$ and $x+z_n\in K$ for all $n$. Thus by (G3)
\begin{align*}
o(1) &= \frac{\cJ(u_n)}{\|u_n\|^2} = \frac12 \| v_n^+\|^2 - \frac12 \| v_n^-\|^2 - \int_{\R^N} \frac{G(x+z_n,u_n(x+z_n))}{\|u_n\|^2} \, dx \\
&\leq \frac12 \| v_n^+\|^2 - \frac12 \| v_n^-\|^2 - \int_{S} \frac{G(x+z_n,u_n(x+z_n))}{\|u_n\|^2} \, dx \\
&\leq \frac12 - \int_{S} \frac{G(x+z_n,u_n(x+z_n))}{\|u_n\|^2} \, dx \\
&=
\frac12 - \int_{S} \frac{G(x+z_n,u_n(x+z_n))}{u_n(x+z_n)^2} v_n(x+z_n)^2 \, dx 
 \to -\infty
\end{align*}
we get a contradiction. Hence $|S|=0$. If $|\R^N \setminus K| = 0$ (e.g. $K = \R^N$), the proof is completed. Otherwise $\supp v \subset \R^N \setminus K$. Thus, by the $\Z^N$-periodicity of $K$, for all $\varphi \in \cC_0^\infty (\R^N)$ such that $|\supp \varphi \cap \supp v | > 0$ there holds
$$
\left| \supp \varphi \cap \supp v \cap (\R^N \setminus K) \right| = | \supp \varphi \cap \supp v | > 0.
$$
Fix $\varphi \in \cC_0^\infty (\R^N \setminus K)$ and let $\varphi_n := \varphi( \cdot - z_n)$. Then
$$
o(1) = \cJ'(u_n)(\varphi_n) = \langle u_n^+, \varphi_n^+ \rangle - \langle u_n^-, \varphi_n^- \rangle  - \int_{\R^N} g(x,u_n)\varphi_n \, dx.
$$
Note that for sufficiently large $n$
\begin{align*}
\int_{\R^N} g(x,u_n)\varphi_n \, dx &= \int_{\R^N} g(x,u_n(x+z_n))\varphi_n(x+z_n) \, dx\\
&= \|u_n\| \int_{\R^N} \frac{g(x+z_n,u_n(x+z_n))}{u_n(x+z_n)} v_n(x+z_n) \varphi \, dx \\
&= \|u_n\| \int_{\supp \varphi} \frac{g(x+z_n,u_n(x+z_n))}{u_n(x+z_n)} v_n(x+z_n) \varphi \, dx \\
&= \|u_n\| \left( \int_{\supp \varphi \cap \supp v} \frac{g(x+z_n,u_n(x+z_n))}{u_n(x+z_n)} v_n(x+z_n) \varphi \, dx + o(1) \right).
\end{align*}
Recall that for a.e. $x \in \supp \varphi \cap \supp v$ we have $|u_n (x+z_n)| \to \infty$ and $|u_n(x+z_n)| \geq a$ for sufficiently large $n$. Since $x+z_n \not\in K$, we have that
$$
\frac{g(x+z_n,u_n(x+z_n))}{u_n(x+z_n)} v_n(x+z_n) \varphi(x)  = \Theta (x+z_n) v_n(x+z_n) \varphi (x) =  \Theta (x) v_n(x+z_n) \varphi (x)
$$
and
$$
\frac{g(x+z_n,u_n(x+z_n))}{u_n(x+z_n)} v_n(x+z_n) \varphi (x)  \to \Theta (x) v (x) \varphi(x) \quad \mbox{for a.e. } x\in \supp \varphi \cap \supp v.
$$
Again, passing to a subsequence we have $v_n(\cdot + z_n) \to v$ in $L^2 (\supp \varphi \cap \supp v)$. Moreover by (G4) and (G5)
$$
\left| \frac{g(x+z_n,u_n(x+z_n))}{u_n(x+z_n)} \right|^2 \leq |\Theta(x+z_n) |^2 \leq |\Theta|_\infty^2,
$$ 
hence
$$
\frac{g(\cdot+z_n,u_n(\cdot+z_n))}{u_n(\cdot+z_n)} \to \Theta \quad \mbox{in} \ L^2 (\supp \varphi \cap \supp v).
$$
In view of the H\"older inequality
$$
\int_{\supp \varphi \cap \supp v} \frac{g(x+z_n,u_n(x+z_n))}{u_n(x+z_n)} v_n(x+z_n) \varphi \, dx \to \int_{\R^N} \Theta(x) v \varphi \, dx.
$$
Thus
\begin{align*}
&\quad \int_{\R^N} \nabla v \nabla \varphi + V(x) v \varphi \, dx = 
\int_{\R^N} \nabla v_n \nabla \varphi_n + V(x) v_n \varphi_n \, dx + o(1) \\ &= \frac{1}{\|u_n\|} \int_{\R^N} g(x,u_n) \varphi_n \, dx + o(1) =  \int_{\R^N} \Theta(x)  v \varphi \, dx + o(1).
\end{align*}
Finally
$$
\int_{\R^N} \nabla v \nabla \varphi + V(x) v \varphi \, dx =  \int_{\R^N} \Theta(x) v \varphi \, dx \quad \mbox{for } \varphi \in \cC_0^\infty (\R^N \setminus K),
$$
and $0$ is an eigenvalue of the operator $-\Delta + V(x) - \Theta(x)$ on $L^2 (\R^N \setminus K)$ with Dirichlet boundary conditions, which is a contradiction with (G6).
\end{proof}

\begin{Prop}\label{Prop_bounded}
Let $\beta>0$. There exists $M_\beta>0$ such that for every $(u_n)\subset X$ satisfying
$$0\leq \liminf_{n\to\infty}\J(u_n)\leq \limsup_{n\to\infty}\J(u_n)\leq \beta$$ and
$$(1+\|u_n^+\|)\J'(u_n)\to 0\hbox{ and }\J'(u_n)(u_n^-)\to 0,$$
there holds $\limsup_{n\to\infty}\|u_n\|\leq M_\beta$.
\end{Prop}

\begin{proof}
Suppose by contradiction that there is $\beta$ such that for any $k \geq 1$ there is sequence $(u_n^k) \subset X$ satisfying
$$0\leq \liminf_{n\to\infty}\J(u_n^k)\leq \limsup_{n\to\infty}\J(u_n^k)\leq \beta$$
and 
$$(1+\|(u_n^k)^+\|)\J'(u_n^k)\to 0,\hbox{ and }\J'(u_n^k)((u_n^k)^-)\to 0,$$
but $\limsup_{n \to \infty} \| u_n^k \| \geq k$. Choose $n(k)$ such that $\| u_{n(k)}^k\| \geq k - 1$. We may assume that $n(k)$ increases when $k$ increases. Then $(u_{n(k)}^k)$ satisfies all assumptions of Lemma \ref{lem:CerBounded}, but is unbounded -- a contradiction.
\end{proof}

\section{Proof of Theorem \ref{Th:Existence} and Theorem \ref{ThMain2}}\label{sect:existence}

From Theorem \ref{ThLink1} we see that there is a Cerami sequence $(u_n) \subset X^+$ for $\tJ$ at the level $c_\cM > 0$ given by \eqref{cm}. Let $v_n := m(u_n) = u_n + w_n \in \cM$, where $w_n \in X^-$. Then $(\cJ(v_n))$ is bounded. Moreover by property (iv) in Section \ref{sec:criticaslpoitth} we obtain
$$
(1+\| v_n^+ \|) \cJ' (v_n) = (1+\|u_n\|) \tJ ' (u_n) \to 0
 $$
and
$$
\cJ'(v_n)(v_n^-) = 0.
$$
Hence, in view of Lemma \ref{lem:CerBounded}, $(v_n) \subset \cM$ is bounded and therefore $(v_n^+) \subset X^+$ is bounded as well. Then $(v_n) \subset \cM$ is a bounded Palais-Smale sequence for $\cJ$.

\begin{altproof}{Theorem \ref{Th:Existence}}
Up to a subsequence we have
\begin{align*}
v_n \weakto v &\mbox{ for some } v \in X, \\
v_n \to v &\mbox{ in } L^t_{\loc} (\R^N) \mbox{ for all } 2 \leq t < 2^*, \\
v_n \to v &\mbox{ a.e. on } \R^N.
\end{align*}
Suppose that 
$$
\sup_{y \in \R^N} \int_{B(y,1+\sqrt{N})} |v_n^+|^2 \, dx \to 0.
$$
From Lions' lemma $v_n^+ \to 0$ in $L^p (\R^N)$. Then
$$
\cJ'(v_n)(v_n^+) = \frac12 \|v_n^+\|^2 - \int_{\R^N} g(x,v_n)v_n^+ \, dx.
$$
Note that by \eqref{eq:gestimates}
\begin{align*}
\int_{\R^N} |g(x,v_n)v_n^+| \, dx &\leq \varepsilon \int_{\R^N} |v_n| |v_n^+| \, dx + c_\varepsilon \int_{\R^N} |v_n|^{p-1} |v_n^+| \, dx \\ &\leq \varepsilon |v_n|_2 |v_n^+|_2 + c_\varepsilon |v_n|_p^{p-1} |v_n^+|_p \leq \varepsilon M + o(1)
\end{align*}
for some $M > 0$ and therefore $\int_{\R^N} g(x,v_n)v_n^+ \, dx \to 0$. On the other hand
$$
| \cJ'(v_n) (v_n^+) | \leq \| \cJ'(v_n) \| \| v_n^+\|  \to 0.
$$
Hence $v_n^+ \to 0$ in $X$ and
$$
0 < c_{\cM} = \lim_{n\to\infty} \cJ(v_n) = \lim_{n\to\infty} \left( - \frac{1}{2} \|v_n^-\|^2 - \int_{\R^N} G(x,v_n) \, dx \right) \leq 0,
$$
which is a contradiction. Hence there is a sequence $(z_n) \subset \mathbb{Z}^N$ such that
$$
\liminf_{n\to\infty} \int_{B(z_n, 1+\sqrt{N})} |v_n^+|^2 \, dx > 0.
$$
Define $w_n := v_n(\cdot - z_n)$. Then
$$
\liminf_{n\to\infty} \int_{B(0, 1+\sqrt{N})} |w_n^+|^2 \, dx  > 0
$$
and $\|w_n\| = \|v_n\|$, so $(w_n)$ is bounded in $X$ and, up to a subsequence
\begin{align*}
w_n &\weakto w \quad \mathrm{in} \ X, \\
w_n &\to w \quad \mathrm{in} \ L^2_{\loc} (\R^N) \mbox{ and in } L^p_{\loc} (\R^N), \\
w_n(x) &\to w(x) \quad \mbox{for a.e. } x \in \R^N,
\end{align*}
and $w^+ \neq 0$, in particular $w \neq 0$. To show that $w$ is a critical point of $\cJ$ take any test function $\varphi \in \cC_0^\infty (\R^N)$ and see that
\begin{align*}
|\cJ'(w_n)(\varphi)| = | \cJ'(v_n)(\varphi(\cdot + z_n) | \leq \| \cJ'(v_n) \| \|\varphi \| \to 0.
\end{align*}
On the other hand
\begin{align*}
\left| \cJ'(w_n)(\varphi) - \cJ'(w)(\varphi) \right| &\leq  \left| \int_{\R^N} \nabla (w_n - w) \nabla \varphi + V(x) (w_n-w) \varphi \, dx \right|\\ &\quad + \left| \int_{\R^N} (g(x,w_n) - g(x,w)) \varphi   \, dx \right|.
\end{align*}
Note that for every measurable set $E \subset \supp \varphi$
\begin{align*}
\int_{E} |g(x,w_n) \varphi |  \, dx &\leq \varepsilon \int_E |w_n \varphi| \, dx + C_\varepsilon \int_E |w_n|^{p-1} |\varphi| \, dx \\
&\leq \varepsilon |w_n|_2 |\varphi \chi_E|_2 + C_\varepsilon |w_n|_p^{p-1} |\varphi \chi_E|_p.
\end{align*}
Hence $\{ g(x,w_n)\varphi \}$ is uniformly integrable on $\supp \varphi$ and therefore
$$
\left| \int_{\R^N} (g(x,w_n) - g(x,w)) \varphi   \, dx \right| \to 0.
$$
In view of the weak convergence $w_n \weakto w$ we obtain
$$
\langle w_n, \varphi^+ \rangle = \langle w_n^+, \varphi^+ \rangle \to 0, \quad \langle w_n, \varphi^- \rangle = \langle w_n^-, \varphi^- \rangle \to 0
$$
and 
$$
\left| \int_{\R^N} \nabla (w_n - w) \nabla \varphi + V(x) (w_n-w) \varphi \, dx \right| \to 0.
$$
Hence $\cJ'(w)(\varphi) = 0$ and $w$ is a solution. Moreover, \eqref{AR2} and Fatou's lemma show that
\begin{align*}
c_\cM &= \liminf_{n\to\infty} \cJ(w_n) = \liminf_{n\to\infty} \left( \cJ(w_n) - \frac12 \cJ'(w_n)(w_n) \right) = \liminf_{n\to\infty} \frac12 \int_{\R^N} g(x,w_n)w_n - 2 G(x, w_n) \, dx \\
&\geq \frac12 \int_{\R^N} g(x,w)w - 2 G(x, w) \, dx = \cJ(w) - \frac12 \cJ'(w)(w) = \cJ(w),
\end{align*}
i.e. 
$$
\cJ(w) \leq c_\cM.
$$
\end{altproof}

Now, recall that the group $G:=\Z^N$ acts isometrically by translations on $X=X^+\oplus X^-$ and $\cJ$ is $\Z^N$-invariant.
Let
$$\cK:=\big\{v\in X^+: (\cJ\circ m)'(u)=0\big\}$$
and suppose that $\cK$ consists of a finite number of distinct orbits. It is clear that $\Z^N$ acts discretely and hence satisfies the condition (G) in Section \ref{sec:criticaslpoitth}.
Then, in view of Lemma \ref{discrete},
$$\kappa:= \inf\big\{\|v-v'\| \ : \ \J'\bigl(m(v)\bigr) = \J'\bigl(m(v')\bigr) = 0, v\ne v'\big\}>0.$$

\begin{Lem}\label{Discreteness}
	Let $\beta\ge c_{\cN}$ and suppose that $\cK$ has a finite number of distinct orbits. If $(u_n),(v_n)\subset X^+$ are two Cerami sequences for $\tJ$ such that
	\begin{eqnarray*}
	&&0\le\liminf_{n\to\infty}\tJ(u_n)\le \limsup_{n\to\infty}\tJ(u_n)\le\beta,\\ &&0\le\liminf_{n\to\infty}\tJ(v_n)\le \limsup_{n\to\infty}\tJ(v_n)\le\beta,
	\end{eqnarray*} and $\liminf_{n\to\infty}\|u_n-v_n\|< \kappa$, then $\lim_{n\to\infty}\|u_n-v_n\|=0$.
\end{Lem}

\begin{proof}
	Let $m(u_n)=u_n+w^1_n$, $m(v_n)=v_n+w^2_n$. Note that $(\tJ(u_n))$ and $(\tJ(v_n))$ are bounded, hence by Proposition \ref{Prop_bounded},
	 $(m(u_n))$, $(m(v_n))$ are bounded.
	We first consider the following case 
	\begin{equation}
	\label{pqconvergence}
	\lim_{n\to\infty}|u_n-v_n|_{p}=0
	\end{equation}
	and  we prove that
	\begin{equation}
	\label{convergence}
	\lim_{n\to\infty}\|u_n-v_n\|=0.
	\end{equation}
Taking into account \eqref{eq:gestimates} we obtain
	\[
	\begin{split}
	\|u_n-v_n\|^2 = \, & \J'(m(u_n))(u_n-v_n)-\J'(m(v_n))(u_n-v_n)\\
	& +\int_{\R} \left( g(x,m(u_n))-g(x,m(v_n)) \right) (u_n-v_n) \,dx\\
	\le \, & o(1)+\int_{\R}\bigl(|g(x,m(u_n))|+|g(x,m(v_n))|\bigr)|u_n-v_n|\,dx\\
	\le \, & o(1)+\varepsilon \int_{\R^N} |m(u_n)| |u_n - v_n| \, dx + \varepsilon \int_{\R^N} |m(v_n)| |u_n-v_n| \, dx \\ &+ c_\varepsilon \int_{\R^N} |m(u_n)|^{p-1} |u_n - v_n| \, dx + c_\varepsilon \int_{\R^N} |m(v_n)|^{p-1} |u_n - v_n| \, dx \\
	\le \, & o(1) + \varepsilon ( |m(u_n)|_2^2 + |m(v_n)|_2^2 ) |u_n-v_n|_2^2 \\
	&+ c_\varepsilon (|m(v_n)|_p^{p-1} + |m(u_n)|_p^{p-1} ) |u_n - v_n|_p \to 0,
	\end{split}
	\]
	which gives \eqref{convergence}.\\
	\indent 
	Suppose now that \eqref{pqconvergence} does not hold. From Lions' lemma, there are $\eps>0$ and a sequence $(y_n)\subset\Z^N$ such that, passing to a subsequence,
	\begin{equation}
	\label{boundedaway}
	\int_{B(y_n,1+\sqrt{N})}|u_n-v_n|^2\,dx\ge\eps.
	\end{equation}
	Since $\J$ is $\Z^N$-invariant, we may assume $y_n=0$. As $(m(u_n)), (m(v_n))$ are bounded, up to a subsequence,
	\begin{equation}
	\label{weakconvergence}
	m(u_n) \rightharpoonup u+w^1\hbox{ and }m(v_n)\rightharpoonup v+w^2\quad\hbox{in }X^+ \oplus X^- 
	\end{equation}
	for some $u,v\in X^+$ and $w^1,w^2\in X^-$. Passing to a subsequence we may assume that $u_n\to u$ and $v_n\to v$ in $L^2_{\loc}(\R^N)$, hence $u\ne v$ according to \eqref{boundedaway}.
	We can easily compute that for any $\varphi \in \cC_0^\infty (\R^N)$
	$$
	\J'(m(u_n))(\varphi) \to \cJ'(u+w^1)(\varphi)	
	$$
	Since $(m(u_n))$ and $(m(v_n))$ are Palais-Smale sequences, one can easily compute that
	\[
	\J'(u+w^1)=\J'(v+w^2)=0.
	\]
	Thus
	\[
	\liminf_{n\to\infty}\|u_n - v_n\| \ge \|u - v\| \ge \kappa,
	\]
	which is a contradiction.
\end{proof}

\begin{altproof}{Theorem \ref{ThMain2}}
Since (I1)--(I8) are satisfied, $\cJ$ is even and $(M)_0^\beta$ holds by Proposition \ref{Prop_bounded} and Lemma \ref{Discreteness} the statement follows directly by Theorem \ref{Th:CrticMulti}.
\end{altproof}

\appendix

\section{Linking approach}\label{sect:appA}

The existence of a nontrivial solution can be also shown by applying a linking-type argument, cf. \cite{BenciRabinowitz,KryszSzulkin,Mederski2016,LiSzulkin}. Define the set
$$
\mathcal{P} := \{ u \in X \setminus X^- \ : \ u^+ \in \cQ \},
$$
where $\cQ$ is the vector space defined in Section \ref{sect:Q}. We shall see that $\mathcal{P}$ joins the linking geometry with the set
$$
\cN_\cQ := \{ u \in \mathcal{P} \ : \ \cJ'(u) |_{\R u \oplus X^-} = 0 \}.
$$
Note that for $K = \R^N$ we take $\cQ = X^+$ and then we have $\cP = X \setminus X^-$, so that $\cP$ joins the linking geometry with the so-called Nehari-Pankov manifold
$$
\cN = \{ u \in X \setminus X^- \ : \ \cJ'(u) |_{\R u \oplus X^-} = 0 \}.
$$
Otherwise $\cN_\cQ$ may be a proper subset of $\cN$.

\begin{Lem}\label{linking}
The functional $\cJ$ has the linking geometry, i.e. there exists $r > 0$ such that
$$
\inf_{u \in X^+, \ \|u\|=r} \cJ(u) > 0,
$$
and for every $u \in \cP$ there is $R(u) > r$ such that
$$
\sup_{\partial M(u)} \cJ \leq \cJ(0) = 0,
$$
where
$$
M(u) = \{ tu+v \ : \ t \geq 0, \ v \in X^-, \ \|tu+v\| \leq R(u) \}.
$$
\end{Lem}
 
\begin{proof}
The first part follows directly from \eqref{eq:gestimates}. Take $u \in \cP$. Observe that
$$
\partial M(u) = \{ tu+v \ : \ v \in X^-, \|tu+v\| = R(u), \ t > 0 \} \cup \{ v \in X^- \ : \ \|v\| \leq R(u) \} =: M_1 \cup M_2.
$$
Obviously, if $v \in X^-$, we have $\cJ(v) \leq 0$. Hence
$$
\sup_{M_2} \cJ \leq 0.
$$
Suppose by contradiction that $\sup_{M_1} \cJ > 0$, i.e. there are $v_n \in X^-$ and $t_n > 0$ such that $\cJ(t_n u + v_n) > 0$ and $\|t_n u + v_n \|\to\infty$. Define
$$
w_n := \frac{t_n u + v_n}{\|t_n u + v_n\|}.
$$
Note that
$$
w_n^+ = \frac{t_n}{\|t_n u + v_n\|} u^+, \quad w_n^- = \frac{t_n u^- + v_n}{\|t_n u + v_n\|} .
$$
Let $s_n := \frac{t_n}{\|t_n u + v_n\|} > 0$. Then
$$
w_n = s_n u^+ + w_n^-.
$$
Obviously $\|w_n\| = 1$ and therefore passing to a subsequence $w_n \weakto w$ and
\begin{align*}
s_n &\to s, \\
w_n^- &\weakto w^-, \\
w_n^- &\to w^- \quad \mbox{in} \ L^2_{\loc} (\R^N)\hbox{ and in }\ L^p_{\loc} (\R^N).
\end{align*}
Moreover
\begin{align*}
0 < \frac{\cJ(t_n u + v_n)}{\|t_n u + v_n\|^2} &= \frac12 \frac{t_n^2 \|u^+\|^2}{\|t_n u + v_n\|^2} - \frac12 \frac{\|t_nu^- + v_n\|^2}{\|t_n u + v_n\|^2} - \int_{\R^N} \frac{G(x,t_n u + v_n)}{\|t_n u + v_n\|^2} \, dx \\
&= \frac12 s_n^2 \|u^+\|^2 - \frac12 \| w_n^-\|^2 - \int_{\R^N} \frac{G(x,t_n u + v_n)}{\|t_n u + v_n\|^2} \, dx.
\end{align*}
If $s = 0$ we have
$$
0 \leq \frac12 \| w_n^-\|^2 + \int_{\R^N} \frac{G(x,t_n u + v_n)}{\|t_n u + v_n\|^2} \, dx < \frac12 s_n^2 \|u^+\|^2 \to 0.
$$
In particular $w_n^- \to 0$ and therefore $\|w_n\| \to 0$ -- a contradiction. Hence $s > 0$ and $su^+ + w^- \in X \setminus X^-$. Since $su^+ \in \cQ$, by Lemma \ref{lem:Q} we get
$$
|\supp (   s u^+ + w^- ) \cap K| > 0.
$$
Hence, in view of the Fatou's lemma and (G3)
\begin{align*}
0 &\leq \frac12 s_n^2 \|u^+\|^2 - \frac12 \| w_n^-\|^2 - \int_{\R^N} \frac{G(x,t_n u + v_n)}{\|t_n u + v_n\|^2} \, dx \\
&= \frac12 s_n^2 \|u^+\|^2 - \frac12 \| w_n^-\|^2 - \int_{\R^N} \frac{G(x,t_n u + v_n)}{|t_n u + v_n|^2} |s_n u^+ + w_n^-|^2 \, dx \\
&\leq \frac12 s_n^2 \|u^+\|^2 - \frac12 \| w_n^-\|^2 - \int_{K} \frac{G(x,t_n u + v_n)}{|t_n u + v_n|^2} |s_n u^+ + w_n^-|^2 \, dx \\
&\leq \frac12 s_n^2 \|u^+\|^2 - \int_{K} \frac{G(x,t_n u + v_n)}{|t_n u + v_n|^2} |s_n u^+ + w_n^-|^2 \, dx \to -\infty
\end{align*}
-- a contradiction, since $|K| > 0$. Hence $\sup_{M_1} \cJ \leq 0$ and the proof is completed.
\end{proof} 
  
For any set $A \subset H^1 (\R^N)$, $I \subset [0,\infty)$ such that $0 \in I$ and a function $h : A \times I \rightarrow H^1 (\R^N)$ we collect the following assumptions inspired by \cite{LiSzulkin}:
\begin{enumerate}
\item[(h1)] $h$ is continuous;
\item[(h2)] $h(u,0) = u$ for all $u \in A$;
\item[(h3)] $\cJ(h(u,t)) \leq \max\{ \cJ(u), -1 \}$ for all $(u,t) \in A \times I$;
\item[(h4)] for every $(u,t) \in A \times I$ there is an open neighbourhood $W$ in $H^1 (\R^N) \times I$ such that the set $\{v-h(v,s) \ : \ (u,t) \in W \cap (A \times I) \}$ is contained in a finite-dimensional subspace of $H^1 (\R^N)$.
\end{enumerate}  
In view of \cite[Theorem 2.1]{Mederski2016}, there exists a Cerami sequence at level $c_{\cP}$, i.e. a sequence $(u_n) \subset H^1 (\R^N)$ such that
$$
\cJ(u_n) \to c_\cP, \quad (1+\|u_n\|) \cJ'(u_n) \to 0,
$$
where
\begin{align}\label{c}
c_{\cP} &:= \inf_{u \in \cP} \inf_{h \in \Gamma(u)} \sup_{u' \in M(u)} \cJ(h(u',1)) > 0,\\ \nonumber
\Gamma(u) &:= \{ h \in \cC (M(u) \times [0,1]): \ h \mbox{ satisfies } (h1)-(h4) \}.
\end{align}
From Lemma \ref{lem:CerBounded} we know that $(u_n)$ is bounded and we may follow the proof of Theorem \ref{Th:Existence} as in Section \ref{sect:existence} and show that there is a critical point $u \neq 0$ of $\cJ$ such that $\cJ(u) \leq c_{\cP}$.

\section*{Acknowledgements}
The authors would like to thank referees for several valuable comments and especially for indicating a gap in the proof of Lemma 5.1 in the original version of the manuscript.
They are also grateful to Prof. Jari Taskinen for pointing out us references \cite{KuchmentBAMS,KuchmentTaskinen}. 
Bartosz Bieganowski was partially supported by the National Science Centre, Poland (Grant No. 2017/25/N/ST1/00531). Jaros\l aw Mederski was partially supported by the National Science Centre, Poland (Grant No. 2017/26/E/ST1/00817).

\end{document}